\documentclass[a4paper,twoside,11pt,reqno]{amsart}
\usepackage{fullpage}
\usepackage[T1]{fontenc}
\usepackage[utf8]{inputenc}

\usepackage{hyperref}
\usepackage[slantedGreeks, partialup, noDcommand]{kpfonts}
\usepackage{graphicx}
\usepackage[mathscr]{euscript}
\usepackage{caption}
\usepackage{tikz}
\usepackage{appendix}

\usetikzlibrary{trees}

\hypersetup{ocgcolorlinks=true,allcolors=testc}
\hypersetup{
	colorlinks   = true,
	citecolor    = black
}
\hypersetup{linkcolor=black}
\hypersetup{urlcolor=black}

\newtheorem{theorem}{Theorem}
\newtheorem{lemma}{Lemma}

\newtheorem{proposition}{Proposition}
\newtheorem{corollary}[theorem]{Corollary}

\newtheorem*{definition*}{Definition}

\theoremstyle{definition}

\theoremstyle{remark}

\renewcommand{\epsilon}{\varepsilon}

\newcommand{\N}{\mathbb{N}}

\renewcommand{\phi}{\varphi}
\newcommand{\R}{\mathbb{R}}

\DeclareMathOperator{\sign}{sign}

\renewcommand{\Pr}[2][]{\mathbb{P}_{#1} \left\{ #2 \rule{0mm}{3mm}\right\}}

\newcommand{\ip}[2]{\langle#1,#2\rangle}

\def \N {\mathbb{N}}

\def \R {\mathbb{R}}

\def \d {\delta}

\def \psitwo {{\psi_2}}

\def\bal{\begin{align*}}
\def\eal{\end{align*}}

\author{Paata Ivanisvili}
\address{(P.~I.) Department of Mathematics, University of California, Irvine, CA 92697, USA}
\email{pivanisv@uci.edu}

\author{Roman Vershynin}
\address{(R.~V.) Department of Mathematics, University of California, Irvine, CA 92697, USA}
\email{rvershyn@uci.edu}

\author{Xinyuan Xie}
\address{(X.~X.) Department of Mathematics, University of California, Irvine, CA 92697, USA}
\email{xinyuax7@uci.edu}

\title{Jackson's inequality on the hypercube}

\thanks{}

\begin{document}

    \maketitle
	
    \begin{abstract}
We investigate the best constant \( J(n,d) \) such that Jackson's inequality
\[
\inf_{\deg(g) \leq d} \|f - g\|_{\infty} \leq J(n,d) \, s(f),
\]
holds for all functions \( f \) on the hypercube \( \{0,1\}^n \), where \( s(f) \) denotes the sensitivity of \( f \). We show that 
 the quantity \( J(n, 0.499n) \) is bounded below by an absolute positive constant, independent of \( n \). This complements Wagner's theorem, which establishes that \( J(n,d)\leq 1 \).  As a first application we show that reverse Bernstein inequality  fails in the tail space $L^{1}_{\geq 0.499n}$ improving over previously known counterexamples in  $L^{1}_{\geq C \log \log (n)}$. As a second application, we show that there exists a function \( f : \{0,1\}^n \to [-1,1] \) whose sensitivity \( s(f) \) remains constant, independent of \( n \), while the approximate degree grows linearly with \( n \). This result implies that the sensitivity theorem \( s(f) \geq \Omega(\deg(f)^C) \) fails in the strongest sense for bounded real-valued functions even when \( \deg(f) \) is relaxed to the approximate degree. We also show that in the regime \( d = (1 - \delta)n \), the bound
\[
J(n,d) \leq C \min\{\delta, \max\{\delta^2, n^{-2/3}\}\}
\]
holds. Moreover, when restricted to symmetric real-valued functions, we obtain \( J_{\mathrm{symmetric}}(n,d) \leq C/d \) and the decay $1/d$ is sharp.  Finally, we present results for a subspace approximation problem: we show that there exists a subspace \( E \) of dimension \( 2^{n-1} \) such that 
$\inf_{g \in E} \|f - g\|_{\infty} \leq s(f)/n$ holds for all \( f \). 
	\end{abstract}
	
	\bigskip
	
	{\footnotesize
		
		}

	 \section{Introduction}
\subsection{Approximation of real-valued functions on the hypercube}  Let $n\geq 1$ be an integer, and let $\{0,1\}^{n}$ be the $n$-dimensional hypercube. Set $[n]:=\{1, \ldots, n\}$.  Any function $f$ on the hypercube admits Fourier--Walsh series representation 
\begin{align*}
    f(x) = \sum_{S \subset [n]} \widehat{f}(S) W_{S}(x) \quad \text{where} \quad W_{S}(x) = (-1)^{\sum_{j\in S} x_{j}}. 
\end{align*}
 Here $x = (x_{1}, \ldots, x_{n}) \in \{0,1\}^{n}$,  $\widehat{f}(S)  = \mathbb{E} f(X) W_{S}(X)$, and $X \sim \mathrm{unif}(\{0,1\}^{n})$. The {\em degree} of $f$ is the minimal number $\mathrm{deg}(f)$ such that $\widehat{f}(S)=0$ for all $|S|>\mathrm{deg}(f)$ where $|S|$ denotes the cardinality of the subset $S \subset \{1, \ldots, n\}$. Clearly $\mathrm{deg}(f)$ conincides with the degree of a multilinear polynomial $g :\mathbb{R}^{n} \to \mathbb{R}$ such that $g=f$ on $\{0,1\}^{n}$. For any $p\geq 1$ let $\|f\|_{p} = (\mathbb{E} |f(X)|^{p})^{1/p}$, where $X \sim \mathrm{Unif}(\{0,1\}^{n})$. If $p=\infty$ we simply set $\|f\|_{\infty}:= \max_{x \in \{0,1\}^{n}} |f(x)|$. 

For any $f :\{0,1\}^{n} \to \mathbb{R}$, and any $d$, $0\leq d \leq n$, we are interested in the best uniform polynomial approximation of $f$ on the hypercube, which is measured by 
\begin{align*}
    E_{d}^{n}(f) \stackrel{\mathrm{def}}{=} \inf_{\mathrm{deg}(g)\leq d} \| f-g\|_{\infty}.
\end{align*}

A recent breakthrough result due to Hao Huang \cite{Hun19}, resolving the long standing {\em sensitivity conjecture}, states that all Boolean functions $f : \{0,1\}^{n} \to \{0,1\}$ satisfy 
\begin{align}\label{Jak00}
    s(f) \geq c_0 \mathrm{deg}(f)^{c},
\end{align}
where $c_0,c>0$ are universal constants. 
In an equivalent form, Huang's result can be stated as a {\em Jackson's inequality on the hypercube}: all Boolean functions $f : \{0,1\}^{n} \to \{0,1\}$ satisfy
\begin{align}\label{Jak1}
    E_{d}^{n}(f) \leq \frac{c_{1}}{d^{c_{2}}} s(f),  
\end{align}
where $s(f)$ denotes the sensitivity of $f$, i.e.
\begin{align*}
    s(f) := \max_{x \in \{0,1\}^{n}} \sum_{j=1}^{n} |f(x)-f(x^{j})|, \quad x^{j} : = (x_{1}, \ldots, x_{j-1},1-x_{j}, x_{j+1}, \ldots, x_{n}) 
\end{align*}
and where $c_{1}, c_{2}>0$ are universal constants, see Subsecton~\ref{help01}.

A related result is Pisier's inequality  due to Wagner \cite{Wag1} which unlike (\ref{Jak1}) holds for all real-valued functions but with a weaker upper bound. More precisely we have 
\begin{equation}    \label{eq: Wagner}
    E_{0}^{n}(f) = \inf_{c \in \mathbb{R}}\|f-c\|_\infty \leq s(f) \quad \text{for all} \quad f :\{0,1\}^{n} \to \mathbb{R}. 
\end{equation}
 
 It is natural to wonder how the {\em approximaton error}  $E_{d}^{n}(f)/s(f)$ behaves for an arbitrary real valued $f : \{0,1\}^{n} \to \mathbb{R}$ and any degree $d \in [0,n]$.

\subsection{Approximation by low-degree polynomials}
 Our first result shows that if $f$ is {\em symmetric}, that is if the value $f(x_{1}, \ldots, x_{n})$ is independent of a permutation of the variables $x_{1}, \ldots, x_{n}$, then the approximation error is $O(1/d)$. 
 
\begin{proposition}[Approximability of symmetric functions by low-degree polynomials] \label{sym-case}
There exists a universal constant $C>0$ such that for any symmetric function $f:\{0,1\}^{n} \to \mathbb{R}$ and any $d \in (0,n]$,  we have 
\begin{align}\label{symup1}
E_{d}^{n}(f) \leq \frac{C}{d} s(f).
\end{align}
Moreover, this bound is optimal. For every $n \in \N$ and any $d \in (0,n]$, there exists a symmetric function $f:\{0,1\}^{n} \to \mathbb{R}$ with $s(f) > 0$ such that 
\begin{align}\label{symlow1}
    E_{d}^{n}(f) \geq \frac{1}{8d} s(f).
\end{align}
\end{proposition}
 
However, an estimate of the type (\ref{Jak1}) fails in the strongest possible sense as soon as we drop the assumption of $f$ being symmetric or Boolean:
 
\begin{theorem}[Inapproximability of general functions by low-degree polynomials] \label{lb-thm}
For any positive $c_1 < 1/2$, there exists $c_{2}>0$ such that for all $n\geq 1$ there exists a function  $f :\{0,1\}^{n} \to \mathbb{R}$ with $s(f) > 0$ satisfying  
\begin{align}\label{dabali}
E^{n}_{c_{1}n}(f) \geq c_2 s(f).
\end{align}
\end{theorem}
The original formulation of Huang's result \eqref{Jak00}, namely
\begin{align}\label{Jak2}
s(f) \geq \sqrt{\mathrm{deg}(f)}
\end{align}
is valid for all Boolean functions, but due to inhomogeneity such estimate can not generalize to real valued functions $f :\{0,1\}^{n} \to [-1,1]$. A simple counterexample is $f(x) = \frac{1}{n}(-1)^{x_{1}+\cdots+x_{n}}$ where $s(f)=2$ while $\mathrm{deg}(f)=n$. 

One may wonder if there is a chance to have an estimate of the form (\ref{Jak2}) for all $f :\{0,1\}^{n} \to [-1,1]$ if we replace $\mathrm{deg}(f)$ by the {\em approximate degree} $\widetilde{\mathrm{deg}}(f)$, which is defined as the smallest number $\ell$ such that $E_{\ell}^{n}(f)\leq 1/3$, where 1/3 here is by convention and can be replaced by any number in $(0,1)$. Notice that the approximate degree of the counterexample above is zero for $n\geq 3$. The proof of Theorem~\ref{lb-thm} implies that the answer to this question is negative. 

\begin{corollary}[No sensitivity bound for real-valued functions via approximate degree] \label{gamo1}
   Let $h$ be any function on the real line, which increases to infinity as $x \to \infty$. The inequality 
   \begin{align*}
       s(f) \geq h \left( \widetilde{\mathrm{deg}}(f) \right)
   \end{align*}
   does not hold for all functions $f : \{0,1\}^{n} \to [-1,1]$ and all $n\geq 1$. 
\end{corollary}

\subsection{Failure of reverse Bernstein inequality in the tail space $L^{1}_{\geq 0.499n}$}
For any $d \in [0,n]$ and  any $p\geq 1$,  let $L_{\geq d}^{p}$ denote the $d$'th tail space equipped with the norm $\|\cdot \|_{p}$ and consisting  of all functions $f$ on the hypercube that can be represented as $f(x) = \sum_{|S|\geq d} a_{S}W_{S}(x)$. Set $\Delta f(x) = \sum_{j=1}^{n} (f(x)-f(x^{j}))$. Theorem~\ref{lb-thm} gives the following 
\begin{corollary}\label{heat11}
    For any $c_{1} \in (0,1/2)$, there exists $c_{2}>0$ that depends only on $c_1$ and such that for any $n\geq 1$ there exists a function $f\in L^{1}_{\geq c_{1}n}$, $f \not\equiv 0$,  with the property 
    \begin{align*}
        \|f\|_1 \ge c_2 \|\Delta f\|_1.
    \end{align*}
\end{corollary}
In \cite{MN14} Mendel and Naor showed that if the reverse Bernstein inequality $\|f\|_p \le \frac{C_{p}}{n} \|\Delta f\|_p$
were to hold  in the tail space $L^{p}_{\geq n/10}$ for vector-valued functions then this would simplify their construction of vector-valued expanders (see Remark 7.5  in \cite{MN14}). They also proved (see \cite[equation 145]{MN14}) that the inequality fails in  the tail space $L^{1}_{ \geq c\log\log n}$.  Corollary~\ref{heat11} shows that one can find a counterexample to the reverse Bernstein inequality even in the smaller space $L^{1}_{\geq cn}$. Before we conclude this section, we point out that for any $p \in (1, \infty),$ the reverse Bernstein inequality in the tail space $L^{p}_{\geq d}$ remains a major open problem, see \cite{HMO17, EI20, EI23} for some partial results.

\subsection{Approximation by low-dimensional subspaces}

In fact, the existence of a {\em poorly} approximable function $f$ in Theorem~\ref{lb-thm} is not due to the nature of the space of polynomials of degree at most $d$. It is just a consequence of the dimension of this space, which equals $\binom{n}{\leq d} = \sum_{j=0}^{d} \binom{n}{j}$. The proof of Theorem~\ref{lb-thm} actually shows that no subspace of this dimension can be used to approximate $f$. In the language of approximation theory, we find a lower bound on the {\em Kolmogorov width} of the set of functions of given sensitivity: 

\begin{theorem}[Inapproximability by low-dimensional subspaces]\label{cor-subapp}
For any $c_1 \in (0, 1/2)$, there exists small $c_{2}>0$ and large $n_{0}>0$ such that for any $n \geq n_{0}$  and any subspace $E$ of dimension at most $\binom{n}{\leq c_{1}n}$ there exists a function $f :\{0,1\}^{n} \to \mathbb{R}$ with $s(f) > 0$ satisfying 
\begin{align*}
\inf_{g \in E}\|f-g\|_{\infty} \geq c_2 s(f).
\end{align*}
\end{theorem} 

It is an interesting problem to understand what happens in Theorem~\ref{lb-thm} in the critical regime where $c_{1}=1/2$. Notice that if $n$ is even, the dimension of the space of all polynomials of degree at most $n/2$ is $\binom{n}{\le n/2} = 2^{n-1}$. In the next theorem we show that there exists a subspace of dimension $2^{n-1}$ in the space of all functions on $\{0,1\}^{n}$ such that the approximation error is at most $s(f)/n$. In particular, it shows that Theorem~\ref{cor-subapp} is sharp, i.e., one cannot relax the assumption $c_{1}<1/2$.  

\begin{theorem}[Approximability by a subspace of half-dimension] \label{subsapp}
There exists a subspace $E$ of dimension $2^{n-1}$ such that
\begin{align*}
\inf_{g \in E} \| f - g \|_\infty < \frac{s(f)}{n}
\end{align*}
holds for all $f :\{0,1\}^{n} \to \mathbb{R}$. 
\end{theorem}

\subsection{Approximation by high-degree polynomials}
Let us revisit the high-degree approximation. How well can we approximate a general function $f :\{0,1\}^{n} \to \mathbb{R}$ by a polynomial of degree $d=(1-\delta)n$ where $\delta \in [0, 1]$ is fixed? In other words, what is the smallest mulitplier $J(n,d)$ that makes Jackson-type inequality 
$$
E_{d}^{n}(f) \le J(n,d) s(f)
$$ 
hold true in this regime? Pisier's inequlity \eqref{eq: Wagner}  with Theorem~\ref{lb-thm} show that if $\delta \in (1/2, 1]$ then
$$
c_1 \le J(n,d) \le 1
$$
where $c_1=c_{1}(\delta)>0$ depends only on $\delta$.
At the opposite end of the spectrum, setting $\d=0$ makes $J(n,d)=0$, since the function $f$ itself can be expressed as a polynomial of degree $d=n$. This makes us wonder whether and how $J(n,d)$ decreases to zero if we let $\d$ decrease to $0$. For example, is it true that $J(n, 3n/4)\to 0$ as $n \to \infty$?  The following result gives a bound of this kind. For $x = (x_{1}, \ldots, x_{n})$ and $y  = (y_{1}, \ldots, y_{n})$ in $\{0,1\}^{n}$ set
\begin{align*}
x\oplus y = (x_{1}+y_{1}\;  \mathrm{mod}(2), \ldots, x_{n}+y_{n} \; \mathrm{mod}(2)).
\end{align*}
 Pick any polynomial $h$ of degree at most $d$ on the real line, which satisfies $\mathbb{E} h(X)=1$ where $X \sim \mathrm{Bin}(n,1/2)$, and let $H(x_{1}, \ldots, x_{n}) = h(x_{1}+\cdots+x_{n})$. 
\begin{theorem}\label{mtavarit}
We have 
\begin{align}\label{mtavaritt}
E_{d}^{n}(f)\leq \max_{x \in \{0,1\}^{n}}| f(x) - \mathbb{E}_{y} f(y)H(y\oplus x)| \leq 3\frac{s(f)}{n} \mathbb{E}X|h(X)|,
\end{align}
where $y \sim \mathrm{Unif}(\{0,1\}^n)$, and $X \sim \mathrm{Bin}(n,1/2)$. 
\end{theorem}

Combining (\ref{mtavaritt}) with a {\em quadrature argument} in Section~\ref{kuadratura1}, we obtain

\begin{corollary}\label{k-bound}
For any $f :\{0,1\}^{n} \to \mathbb{R}$, and any $d \in [0,n]$, we have 
\begin{align}\label{kravup}
E_{d}^{n}(f) \leq  3 \frac{k_{n, \lfloor d/2 \rfloor +1}}{n} s(f) 
\end{align}
where $k_{n,\ell}$ is the smallest positive root of degree $\ell$ Kravchuk polynomial, i.e., polynomials orthogonal with respect to the measure  $\mathrm{Bin}(n, 1/2)$.
\end{corollary}

Alternatively, combining \eqref{mtavaritt} with a specific choice of $h$ improves the bound (\ref{kravup}) in the regime $\delta \leq \max\{\delta^{2}, n^{-2/3}\}$, where $\delta = 1-d/n$.  

\begin{corollary} \label{pr-bound}
For any $f :\{0,1\}^{n} \to \mathbb{R}$, and any $d$, $0 \leq d \leq n$, we have 
\begin{align*}
E_{d}^{n}(f) \leq  3 \left(1-\frac{d}{n} \right) s(f).
\end{align*}
\end{corollary}

Corollaries~\ref{k-bound} and \ref{pr-bound}, together with additional results on the smallest positive roots of Kravchuk polynomials, provide us with the following bound.

\begin{corollary}[Approximation by high-degree polynomials] \label{genupper}
For any $f :\{0,1\}^{n} \to \mathbb{R}$, and any $d \in [0,n]$ we have 
\begin{align}\label{upperbb1}
E_{d}^{n}(f) \leq  
C \min\{\delta,  \max\{\delta^2, n^{-2/3}\}\} s(f),
\end{align}
where $C>0$ is an absolute constant, and $\delta=1-d/n$.
\end{corollary}

\subsection{Approximation of random Boolean functions}

Our results so far have been about {\em arbitrary} real-valued functions on the hypercube. Next we consider the same approximation problem for a {\em random} Boolean function. We say that $f : \{0,1\}^{n} \to 
 \{-1,1\}$ is a random Boolean function if  $\{f(x)\}_{x \in\{0,1\}^{n}}$ are independent identically distributed symmetric $\pm 1$ Bernoulli random variables. The following result is due to \cite{OS08} which shows that $f_{\leq d}$, where $d$ is slighly greater than $n/2$, is a good approximation in $L^{\infty}$ to $f$ with high probability. Set 
\begin{align*}
    f_{>d}(x) := \sum_{|S|> d} \widehat{f}(S) W_{S}(x),
\end{align*}
and denote $f_{\leq d} :=f-f_{>d}$. It is clear that $\inf_{\mathrm{deg}(g)\leq d}\|f-g\|_{2}$ is achieved on $g=f_{\leq d}$.

\begin{theorem}[Approximability of random Boolean functions by polynomials of degree $>n/2$] \label{rocco}
If $d \ge n/2 + C\sqrt{n \log n}$, then 
a random Boolean function $f:\{0,1\}^n \to \{-1,1\}$ can be uniformly $1/n^{10}$-approximated by a polynomial of degree $d$. Specifically, we have
\begin{align}\label{ryan1}
\| f_{>d} \|_\infty \le \frac{1}{n^{10}}
\end{align}
with probability at least $1-2^{-n}$.
\end{theorem}
The original proof of Theorem~\ref{rocco} in \cite{OS08} did not use $f_{\leq d}$ as an approximating polynomial, however, their argument still implies the bound (\ref{ryan1}). In Subsection~\ref{roc01} we will present a direct and simpler proof of Theorem~\ref{rocco}.

It is an interesting and open porblem, Wang--Williams conjecture, see also \cite[Conjecture 8.2]{BV19}, what should be the correct term instead of $C\sqrt{n \log(n)}$ in Theorem~\ref{rocco}.
In general the degree $n/2$ is needed in Theorem~\ref{rocco}, i.e., half of the Boolean functions cannot be uniformly well  approximated by polynomials of degree at most $n/2$. This result is due to \cite{Anth}.  Since the proof is short we decided to present it in Subsection~\ref{cnobili} for completeness. 

\begin{theorem}[Degree $n/2$ is a sharp threshold]\label{thres1}
At least $50\%$ of all Boolean functions $f:\{0,1\}^n \to \{-1,1\}$ satisfy 
$$
\inf_{\deg(g) \le n/2 } \|f-g\|_\infty \ge 1.
$$
Equivalently, at least $50\%$ of all Boolean functions are not sign-representable by any PTF of degree $\leq n/2$, where $PTF$ stands for polynomial threshold function, i.e., sign of degree at most $d$ real-valued polynomial on the hypercube.
\end{theorem}


\section{Proofs}

\subsection {Proof of Proposition~\ref{sym-case}}
First we verify the upper bound (\ref{symup1}) in Proposition~\ref{sym-case}. Since $f$ is symmetric, the value of $f$ at any point $x = (x_1,\ldots,x_n) \in \{0,1\}^n$ is uniquely determined by the Hamming weight $x_1+\cdots+x_n$. In other words, for any symmetric function $f: \{0,1\}^n \to \R$ there exists a function $\phi: \{0,\ldots,n\} \to \R$ such that 
\begin{equation} \label{eq: f phi}
f(x_1,\ldots,x_n) = \phi (x_1+\cdots+x_n) 
\quad \text{for all } x \in \{0,1\}^n.
\end{equation}
Conversely, for any function $\phi: \{0,\ldots,n\} \to \R$, the function $f: \{0,1\}^n \to \R$ defined by \eqref{eq: f phi} is symmetric. This observation allows us to use symmetry to greatly simplify the computations of the approximation error and sensitivity.

\begin{lemma}[Approximation error for symmetric functions] \label{sym-lbd} 
    For any symmetric $f : \{0,1\}^{n} \to \mathbb{R}$ we have
    \begin{align*}
        \inf_{\mathrm{deg}(g)\leq d} \| f-g\|_{\infty} 
        = \inf_{ \substack{\mathrm{deg}(g)\leq d \\ g \text{ is symmetric}} } \|f-g\|_{\infty} 
        = \inf_{\mathrm{deg}(h)\leq d} \max_{k \in \{0,\ldots,n\}}| \phi(k)-h(k)|,
    \end{align*}
    where $\phi: \mathbb{R} \to \mathbb{R}$ is a function that satisfies \eqref{eq: f phi}.
\end{lemma}
\begin{proof}
To check the first equality, recall that by definition, the left hand-side term is less or equal than the right hand-side term. For the reverse direction, for any $g$, we can obtain a symmetric function with smaller error
\begin{align*}
\tilde{g}(x):=\frac{1}{n!}\sum_{\sigma \in S_n} g(x_{\sigma(1)}, \ldots, x_{\sigma(n)}),
\end{align*}
where $S_n$ is the symmetry group of $[n]$. Since
\begin{align*}
\| f(x) - \tilde{g}(x) \|_{\infty} \leq \frac{1}{n!}\sum_{\sigma \in S_n} \| f(x) - g(\sigma(x)) \|_{\infty} = \| f(x) - g(x)\|_{\infty},
\end{align*}
where the last equality follows from the fact that $f$ is symmetric and $f(x) = f(x_{\sigma(1)}, \ldots, x_{\sigma(n)})$.

For the right hand-side equality, we again notice that any symmetric polynomial $g$ on the hypercube $\{0,1\}^{n}$ of degree at most $d$ can be written as $h(x_{1}+\cdots+x_{n})$, where $h$ is a polynomial on the real line of degree at most $d$. 
\end{proof}

\begin{lemma}[Sensitivity of symmetric functions] \label{sym-fact}
    For any symmetric $f :\{0,1\}^{n}\to \mathbb{R}$ we have
    \begin{align*}
        \frac{n}{2} \max_{k \in [n]}|\phi(k) - \phi(k-1)| 
        \leq s(f)	
        \leq n \max_{k \in [n]}|\phi(k) - \phi(k-1)|,
    \end{align*}
    where $\phi: \mathbb{R} \to \mathbb{R}$ is a function that satisfies \eqref{eq: f phi}.
\end{lemma}

\begin{proof}
The right hand-side inequality follows from the definition of $s(f)$. For the left hand-side inequality, let $m$ be the number that attains the maximum {\em jump}, i.e., 
\begin{align*}
|\phi(m) - \phi(m-1)|=\max_{k \in [n]}|\phi(k) - \phi(k-1)|.
\end{align*}
Pick a vector $x \in \{0,1\}^n$ with $|x|=m$. We have 
\begin{align*}
s(f)(x) = \sum_{i=1}^n |f(x) - f(x^i)| \geq  m \max_{k \in [n]}|\phi(k) - \phi(k-1)|.
\end{align*}
Similarly, we obtain that
\begin{align*}
s(f)(x^j) \geq (n-m+1) \max_{k \in [n]}|\phi(k) - \phi(k-1)|,
\end{align*}
where $j$ is a coordinate at which $x_{j}=1$. Thus 
\begin{align*}
 s(f) \geq \max(s(f)(x),s(f)(x^j))\geq \frac{n}{2} \max_{k \in [n]}|\phi(k) - \phi(k-1)|.
\end{align*}
\end{proof}

\begin{lemma}   \label{lem: ak}
For any real numbers $a_0,\ldots,a_n$, we have 
\begin{align*}
\inf_{\mathrm{deg}(h) \leq d } \max_{k \in \{0,\ldots,n\}} |a_k - h(k) | 
\leq \frac{Cn}{d} \max_{k \in [n]} |a_k - a_{k-1}|
\end{align*}
\end{lemma}
\begin{proof}
For any $v : [a,b] \to \mathbb{R}$, let
\begin{align*}
w(v,t) = \sup_{x,y \in [a,b], |x-y| \leq t} |v(x)-v(y)|,
\end{align*}
be the modulus of continuity of $v$. 
It is straightforward to verify that $t \mapsto w(v,t)$ is non-decreasing, and  $t \mapsto w(v,t)$ is sub-additive, i.e., $w(v,t_1+t_2) \leq w(v,t_1) + w(v,t_2)$. Therefore, it follows that 
\begin{equation}    \label{eq: w scaling}
    w(v,ta)\leq (t+1) w(v,a) \quad \forall t>1, \forall a>0.
\end{equation}
Jackson's theorem for algebraic polynomials \cite{Jac12} states that $\forall v \in C([a,b])$,
\begin{align*}
\inf_{\mathrm{deg}(h) \leq d} \max_{x \in [a,b]} | v(x) - h(x) | \leq C w(v, \frac{b-a}{2d}).
\end{align*}
Take $[a,b]=[0,n]$. Set $v(t)$ to be constant $a_k$ on each interval $[k-\frac{1}{3}, k+ \frac{1}{3}]$ (if $k=0$ or $n$, then the interval would be $[0, \frac{1}{3}]$ and $[n-\frac{1}{3},n]$ respectively), and interpolate by a linear function on the remaining parts the domain so that $v(t)$ is continous.  Then, clearly 
\begin{align*}
w(v,\frac{1}{3})=\max_{k \in [n]}|a_k - a_{k-1}|.
\end{align*}
Thus, using \eqref{eq: w scaling} and the fact that $d \leq n$, we have 
\begin{align*}
w(v,\frac{n}{2d}) \leq \left( \frac{3n}{2d}+1\right) w(v,\frac{1}{3}) \leq \frac{3n}{d} \max_{k\in [n]}|a_k -a_{k-1}|.
\end{align*}
Then the lemma follows from 
\begin{align*}
\inf_{\mathrm{deg}(h)\leq d} \max_{k \in \{0,\ldots,n\}}|a_k - h(k)| \leq \inf_{\mathrm{deg}(h) \leq d} \max_{x \in [0,n]} |v(x) - h(x)| \leq \frac{Cn}{d} \max_{k \in [n]}|a_k - a_{k-1}|.
\end{align*}
\end{proof}

Now the upper bound in Proposition~\ref{sym-case} can be proved as follows:
we have 
\begin{align*}
    E_d^n 
    &= \inf_{\mathrm{deg}(g)\leq d} \|f-g\|_{\infty} \\
    &= \inf_{\mathrm{deg}(h)\leq d} \max_{k \in \{0,\ldots,n\}}|\phi(k)-h(k)|
        \quad \text{(by Lemma~\ref{sym-lbd})} \\
    &\le \frac{Cn}{d} \max_{k \in [n]} |\phi(k) - \phi(k-1)|
        \quad \text{(by Lemma~\ref{lem: ak})} \\
    &\le \frac{2C}{d} s(f)
        \quad \text{(by Lemma~\ref{sym-fact})}.
\end{align*}

The proof of the lower bound in Proposition~\ref{sym-case} is based on a general lower bound of the Kolmogorov width of a spaces $C(K)$. Such a result was proved by G.~G.~Lorenz in \cite[Theorem~3]{Lor60} by using a ``hat-packing'' argument similar to the one in the proof of Theorem~\ref{lb-thm}. A simplified version of Lorenz's result can be stated as follows:

\begin{theorem}[Lorenz's bound]   \label{thm: Lorenz}
    Consider a metric space $(K,d)$ that consists of $k+1$ points. Suppose that all pairwise distances between these points are bounded below by $a>0$. Let $H$ be a $k$-dimensional linear space of real-valued functions on $K$. Then there exists a function $\phi: K \to \R$ that is $1$-Lipschits, i.e. satisfying $|\phi(x)-\phi(y)| \le d(x,y)$ for all $x,y \in K$, which satisfies 
    $$
    \|\phi-h\|_\infty \ge \frac{a}{4} \quad \text{for all } h \in H.
    $$
\end{theorem}

To prove the lower bound in Proposition~\ref{sym-case}, we apply Lorenz's Theorem~\ref{thm: Lorenz} for the metric space $K=\{0,\ldots,n\}$ equipped with the usual distance, and for the space $H$ of (restrictions of) univariate polynomials on $K$ whose degree is bounded by $d$. Then $k := \dim(H) = d+1$. Placing the points at equal distances, we can find $k+1$ points in $K$ whose all pairwise distances are at least $a = \lfloor n/(k+1) \rfloor = \lfloor n/(d+2) \rfloor$. Therefore Lorenz's Theorem~\ref{thm: Lorenz} yields the existence of a $1$-Lipschitz function $\phi$ on $K$ that satisfies 
$$
\inf_{\mathrm{deg}(h)\leq d} \max_{k \in K} |\phi(k)-h(k)| \ge \frac{a}{4} \ge \frac{n}{8d}.
$$
Recalling Lemma~\ref{sym-lbd} and the representation of symmetric functions in \eqref{eq: f phi}, we conclude that there exists a symmetric function $f: \{0,1\}^n \to \R$ such that 
$$
E_d^n(f) = \inf_{\mathrm{deg}(g)\leq d} \|f-g\|_{\infty} \ge \frac{n}{8d}.
$$
On the other hand, Lemma~\ref{sym-fact} and the fact that $\phi$ is $1$-Lipschitz imply
$$
s(f) \le n \max_{k \in [n]}|\phi(k) - \phi(k-1)| \le n. 
$$
Combining these two bounds concludes the proof of the lower bound in Proposition~\ref{sym-case}.

\subsection{Proof of Theorem~\ref{lb-thm}}
The main idea of the proof is to construct a class of low sensitivity functions that is {\em complex} enough, so that some of them can not be approximated efficiently. 

We construct such class of functions by considering the maximal number of disjoint {\em smooth} hamming balls (the {\em hat} functions defined in the latter sentence) that we can pack into the cube. More specifically, for $x \in \{0,1\}^n$, let $h_x$ be a 
{\em hat} function centered at the point $x$ with radius $m$, which is defined by assigning $\frac{m-k}{m}$ to each point with hamming distance $k$ to the center if $k < m$, and zero otherwise. Let $K$ be a maximal packing with distance $2m$ on the hypercube $\{0,1\}^n$, i.e., a subset of $\{0,1\}^n$ with maximal number of elements such that the hamming distance between any two  point in the subset is strictly greater than $2m$. Consider the class of functions
\begin{align*}
F := \left\{ \sum_{x \in K} \varepsilon_x h_x \; | \;  \varepsilon_x \in \{1,-1\}, \forall x \in K \right\}.
\end{align*}
The cardinality of $F$ is $2^{P(K,2m)}$, where $P(K, 2m)$ denotes the packing number of $K$ with distance $2m$.  Notice that 
\begin{align*}
s(f) = \frac{n}{m} \quad \text{for all} \quad f \in F. 
\end{align*} 
Next we show that some of elements in $F$ cannot be approximated by low degree polynomials. Fix $d \in \{1,\ldots,n\}$. The polynomials of degree at most $d$ form  linear space
$$
L_{d} \coloneqq \left\{ g:\; \deg(g) \le d \right\}
\quad \text{with} \quad
\dim(L_{d}) = \binom{n}{\le d} \eqqcolon N,
$$
where we used the notation $\binom{n}{\leq k} := \sum_{j=0}^{k}\binom{n}{j}$. 

For each $x \in \{0,1\}^n$, consider the central hyperplane 
$$
H_x \coloneqq \{g \in L_d:\; g(x)=0\}.
$$
There are $ 2^n$ such hyperplanes which partition the space $L_d$ into regions (here we think every function on the cube as a vector in $\mathbb{R}^{2^n}$, and $L_d \setminus (\cup_x h_x)$ is a union of connected components, which are called regions here). Notice that each region is the collection of polynomials $g \in L_d$ that have the same sign pattern $\sign(g)$, i.e., $\mathrm{sign}(g_{1})=\mathrm{sign}(g_{2})$ if and only if $g_{1}, g_{2}$ are in the same connected component. Thus, the number of all possible sign patterns $\sign(g)$ equals the number of regions. On the other hand, for $f \in F$, in order to have $\inf_{g \in L_{d}} \| f-g \|_\infty < 1 = c_2 s(f)$, one must have $\sign(g)=\sign(f)$ on $K$. Thus, the number of $f \in F$ that satisfy this inequality is less or equal than the number of all possible sign patterns.

Schlaffli's formula in the hyperplane arrangement literature \cite[Lemma 2.2]{BV19} states that the number of sign patterns $S$ generated by polynomials of degree at most $d$ is at most $2\binom{2^n -1}{\leq N -1}$ which, in turn, is at most 
$$
 \binom{2^{n}}{\leq N}
$$
whenever $d \leq n/2$ because $ \binom{2^{n}}{N} =  \frac{2^{n}}{N}\binom{2^{n}-1}{N -1}$, and $\frac{2^{n}}{N} = \frac{2^{n}}{\binom{2^{n}}{\leq d}}\geq 2$.

Thus if $\inf_{g \in L_{d}} \|f-g\|_{\infty}<c_{2}s(f)$ holds for all $f \in F$ then we must have 
\begin{align}\label{cina}
    2^{P(K, 2m)}\leq \binom{2^n}{\leq \binom{n}{\leq d}}. 
\end{align}
Let us show that this inequality is not possible in the regime $d=c_{1}n,$  where $c_{1}<1/2$ is fixed, and $n$ is large. 
Indeed, consider binary entropy function $E(p) = p \log_{2}(1/p)+(1-p)\log_{2}(1/(1-p))$. This is a concave function on $[0,1]$, symmetric with respect to $p=1/2$, and such that $E(0)=E(1)=0,$ and $E(1/2)=1$. Let $m=c_{2}n$, where $c_{2} \in (0,1/4)$ is so small that $E(2c_{2})+E(c_{1})-1<0$. Pick any $n >\frac{1}{1-E(c_{1})}$. By \cite[Lemma 16.19]{FlGr} we have $\binom{n}{\leq k} \leq 2^{nE(k/n)}$, whenever $k/n \in (0, 1/2]$. Thus
\begin{align}\label{inf1}
    \binom{2^n}{\leq \binom{n}{\leq d}} \leq 2^{2^{n} E\left(\frac{\binom{n}{\leq c_{1}n}}{2^{n}}\right)} \leq 2^{2^{n}E(2^{n(E(c_{1})-1)})}
\end{align}
where in the second inequality we  used \cite[Lemma 16.19]{FlGr},  the fact that $2^{n(E(c_{1})-1)} \in (0, 1/2)$, and $E(t)$ is nondecreasing on $(0,1/2].$

Replacing volume by cardinality in \cite[Proposition 4.2.12]{Vershynin18} we have 
\begin{align}\label{roman}
  2^n  / { n \choose \leq 2m } \leq P(K,2m),
\end{align}
which is also addressed in \cite[Exercise 4.1.16]{Vershynin18}.
Therefore validity of the inequality (\ref{cina}) implies 
\begin{align*}
    1 \leq \binom{n}{\leq 2m} E(2^{n(E(c_{1})-1)}) \leq 2^{n E(2c_{2})} E(2^{n(E(c_{1})-1)}) \leq  2\cdot 2^{n(E(2c_{2})+E(c_{1})-1)} n(1-E(c_{1})), 
\end{align*}
where in the last inequality we used the simple fact that $E(t) \leq 2 t \log_{2}(1/t)$ for all $t \in (0,1/2]$.
On the other hand the inequality $1 \leq 2\cdot 2^{n(E(2c_{2})+E(c_{1})-1)} n(1-E(c_{1}))$ fails for $n$ sufficiently large (say all $n\geq n_{0}(c_{1})$) due to the assumption $E(2c_{2})+E(c_{1})-1<0$.  However, notice that in the beginning of the proof we could assume without loss of generality that $n \geq n_{0}(c_{1})$ because for small finite number of $n$'s with $1\leq n \leq n_{0}(c_{1})$ we can just take any function $f$ of degree exactly $n$ and in this case we trivially have $\inf_{\mathrm{deg}(g)\leq c_{1}n} \|f-g\|_{\infty}\geq c_{2}(n)s(f)$ holds for some $c_{2}(n)>0$. This finishes the proof of the theorem. 

\subsection{Proof of Corollary~\ref{gamo1}}
The example $f : \{0,1\}^{n} \to [-1,1]$ constructed in the proof of Theorem~\ref{lb-thm} for sufficienlty large $n$  has the property  $s(f) = 1/c_{2}$, and $E_{c_{1}n}^{n}(f)\geq 1$. Clearly this implies $\widetilde{\mathrm{deg}(f)} \geq c_{1}n$, and hence $s(f) \geq h(\widetilde{\mathrm{deg}(f)})$ is not possible as $n \to \infty$. 

\subsection{Proof of equiavlence (\ref{Jak00}) and (\ref{Jak1})}\label{help01}

First let us show the implication (\ref{Jak00}) implies (\ref{Jak1}). Assume $s(f) \geq c_{0} (\mathrm{deg}(f))^{c}$ holds with some universal constants $c_{0}, c_{1}>0$. We claim that $E_{d}^{n}(f) \leq \frac{1}{c_{0}d^{c}}s(f)$ holds  for all Boolean $f$.  Indeed, if $\mathrm{deg}(f)=0$ then the claim is trivial. Assume $\mathrm{deg}(f)\geq d$,  then $\frac{1}{c_{0}d^{c}}s(f)\geq 1$. Since $E_{d}^{n}(f)\leq 1$ the claim follows. 

Next let us show the implication (\ref{Jak1}) implies (\ref{Jak00}). Assume $E_{d}^{n}(f) \leq \frac{c_{1}}{d^{c_{2}}}s(f)$ holds with some universal constants $c_{1}, c_{2}>0$. Choose $d=\widetilde{\mathrm{deg}}(f)-1$. Then we obtain 
\begin{align}\label{dax1}
\frac{1}{3} \leq  \frac{c_{1}}{(\widetilde{\mathrm{deg}}(f)-1)^{c_{2}}}s(f).  
\end{align}
Since $\mathrm{deg}(f) \leq c  \widetilde{\mathrm{deg}}(f)^{8}$, see \cite{NS93}, we obtain (\ref{Jak00}). In this argument without loss of generality we  assume $\widetilde{\mathrm{deg}}(f)>100$ this assumption avoids the issue in (\ref{dax1}) when $\widetilde{\mathrm{deg}}(f)=1$.

\subsection{Proof of Corollary~\ref{heat11}}
Since $s(f)(x) \geq |\Delta f|(x)$ it follows that the inequality (\ref{dabali}) in Theorem \ref{lb-thm} holds true if we replace $s(f)$ by $\| \Delta f \|_\infty$. By the duality argument presented in \cite[Theorem 1]{EI23} (in our case we use Laplacian instead of the discrete gradient), the estimate $E_{c_{1}n}^{n}(f)\geq c_{2} \| \Delta f\|_{\infty}$ implies
$$
\sup_{f(x) = \sum_{|S|> c_1 n} \widehat{f}(S)W_{S}(x)}\frac{\| f \|_1}{\|\Delta f\|_1} \geq C>0,
$$
for some absolute constant $C$ depending on $c_{1}$.

\subsection{Proof of Theorem~\ref{cor-subapp}}
The {\em sign pattern argument} presented in the proof of Theorem~\ref{lb-thm} works verbatim for any subspace of the same dimension $\binom{n}{\leq c_{1}n}$ where $n\geq n_{0}(c_{1})$.


\subsection{Proof of Theorem~\ref{subsapp}}
Without loss of generality, we can assume that $s(f)=1$. Every such function is {\em approximately harmonic}: its value at any point $x$ is within $1/n$ from the arithmetic mean of the values at the neighbors of $x$. This follows from 
\begin{align}
	|f(x) - \frac{1}{n} \sum_{i \in [n]} f(x^{i}) | \leq \frac{1}{n} | \sum_{i \in [n]} f(x) -f(x^i)| \leq \frac{1}{n} s(f). 
\end{align}

The cube consists of even points and odd points (according to the number of ones). Define $E$ as the set of  all {\em odd-harmonic functions} -- those whose value at every odd point is the arithmetic mean of its neighbors (which are all even). The dimension of $E$ is $2^{n-1}$ since it is defined by $2^{n-1}$ linear constrains -- one per every odd point.

Pick any function $f$ with $s(f)\leq 1$ and choose  $g \in E$ that coincides with $f$ at all even points. Since $g \in E$, the value of $g$ at each odd point $x$ is the arithmetic mean of the values of $g$ at the neighbors of $x$. By the fact that $f$ is {\em approximately harmonic}, this value is within $1/n$ from $f$.

\subsection{Proof of Theorem~\ref{mtavarit}} 
The first inequality in (\ref{mtavaritt}) is trivial since the convolution $\mathbb{E}_{y} f(y)H(y\oplus x)$ is a polynomial of degree at most $d$.  We verify the second inequality
\begin{align*}
\max_{x \in \{0,1\}^{n}}| f(x) - \mathbb{E}_{y} f(y)H(y\oplus x)| \leq 3\frac{s(f)}{n} \mathbb{E}X|h(X)|,
\end{align*}
where $y \sim \mathrm{Unif}(\{0,1\}^n)$, and $X \sim \mathrm{Bin}(n,1/2)$. Let $x^{*} \in \{0,1\}^{n}$ be such that $\max_{x}| f(x) - \mathbb{E}_{y} f(y)H(y\oplus x)| =| f(x^{*}) - \mathbb{E}_{y} f(y)H(y\oplus x^{*})|$. Set $F(x) = f(x\oplus x^{*})$. Notice that $s(F)=s(f)$, and 
\begin{align*}
&\max_{x}| f(x) - \mathbb{E}_{y} f(y)H(y\oplus x)| = | F(0, \ldots, 0) - \mathbb{E}_{y} F(y\oplus x^{*})H(y\oplus x^{*})|=\\
& | F(0, \ldots, 0) - \mathbb{E}_{y} F(y)H(y)| \leq  \mathbb{E}|F(0, \ldots, 0)-F(y)| |H(y)| =\\
& \frac{1}{2^{n}}\sum_{k=0}^{n}|h(k)| \sum_{|y|=k} | F(0,\ldots, 0)-F(y)|.
\end{align*}
To complete the proof of the theorem it suffices to show
\begin{lemma}
For each $k=1, \ldots, n$ we have
\begin{align*}
\sum_{|y|=k} | F(0,\ldots, 0)-F(y)| \leq  3 s(F) \binom{n-1}{k-1}.
\end{align*}
\end{lemma}
\begin{proof}
Without loss of generality assume $s(F)= 1$. 

Consider the case $k \leq \frac{n}{2}$. We have 
\begin{align*}
\sum_{|y|=k} | F(0,\ldots, 0)-F(y)| \leq  \frac{1}{k!}\sum_{|y|=k}\sum_{\mathrm{paths} (a_{0}-a_{1}-\cdots-a_{k})} |F(a_j)-F(a_{j+1})|,
\end{align*}
where inner sum runs over all {\em monotone paths} over vertices $(0,\ldots, 0)=a_{0}, .., a_{k}=y$ joining the points $(0, \ldots, 0)$ and $y$. Here $\|a_{j+1}- a_{j}\|_{\ell_{1}^{n}}=1$ and $a_{j+1}\succ  a_{j}$ i.e., $(a_{j+1})_{k} \geq (a_{j})_{k}$ for all $k=1, \ldots, n$.  
We claim that 
\begin{align}\label{symsum}
\frac{1}{k!}\sum_{|y|=k}\sum_{\mathrm{paths} (a_{0}-a_{1}-\cdots-a_{k})} |F(a_j)-F(a_{j+1})| = \sum_{|x|\leq k-1}p_{|x|}\, s(F)(x),
\end{align}
where 
\begin{align*}
p_{\ell} := \frac{\ell! (n-\ell)!}{k! (n-k)! (n+1)} - \frac{(-1)^{k+\ell}}{n+1}, \quad \ell=0, \ldots, k.
\end{align*}
Indeed, a direct calculation shows
$$
p_{\ell}+p_{\ell+1} = \frac{\ell! (n-\ell-1)!}{k! (n-k)!}=\frac{ \ell! (k-\ell-1)!}{k!} \binom{n-\ell-1}{k-\ell-1} 
$$
for all $\ell=0, \ldots, k-1$, and $p_{k}=0$. For each $y \in \{0,1\}^{n}$, $|y|\leq k-1$, the term $|F(y)-F(y^{i})|$ appears $p_{\ell}+p_{\ell+1}$ times in the right hand side of (\ref{symsum}), where $|y|=\ell$ and $|y^{i}|=\ell+1$. Let us count the total number of paths passing through the edge $[y,y^{i}]$  in the left hand side of (\ref{symsum}).   We enter the vertex $y$ in  $\ell!$  ways, and exit the vertex $y^{i}$ in $(k-\ell-1)! \binom{n-\ell-1}{k-\ell-1}$ ways, and this gives  $\ell! (k-\ell-1)!\binom{n-\ell-1}{k-\ell-1}$ total number of paths.

Next we show 
\begin{align*}
\sum_{\ell=0}^{k}|p_{\ell}| \binom{n}{\ell} \leq 2\binom{n-1}{k-1}.
\end{align*}
The inequality simplifies to 
\begin{align*}
\sum_{\ell=0}^{k}\left| 1- (-1)^{k+\ell} \frac{\binom{n}{\ell}}{\binom{n}{k}}\right| \leq \frac{2k(n+1)}{n}.
\end{align*}
The inequality follows from the fact that $\binom{n}{\ell}\leq \binom{n}{k}$ for all $\ell=0, \ldots, k$ and all $k\leq n/2$. 

Next, let us consider the case $k \geq n/2$. The main result in \cite{Wag1} says that $
|F(0, \ldots, 0)-F(1, \ldots, 1)|\leq s(F)$. Therefore, since $\sum_{|y|=k}1  = \binom{n}{k} \leq 2 \binom{n-1}{k-1}$ it suffices to show 
\begin{align*}
\sum_{|y|=k}|F(1, \ldots, 1)-F(y)|\leq \binom{n-1}{k-1}.
\end{align*}
On the other hand for $\tilde{F}(x) = F(x \oplus (1, \ldots, 1))$ we have $s(F)=s(\tilde{F})$, and 
\begin{align*}
\sum_{|y|=k}|F(1, \ldots, 1)-F(y)| = \sum_{|w|=n-k}|\tilde{F}(0, \ldots, 0) - \tilde{F}(w)| \leq \binom{n-1}{n-k-1}\leq \binom{n-1}{k-1}.
\end{align*}
\end{proof}

\subsection{Proof of Corollary~\ref{k-bound}}\label{kuadratura1}
We need the following lemma which follows from \cite[Theorem 2.1 equation (2.7)]{DIM24}. 
\begin{lemma} \label{lemma-int} Let $N$ be an odd integer, and let $\gamma(t)$ denote the moment curve in $\mathbb{R}^{N}$ on the interval $[a,b]$, i.e., $\gamma(t)=(t,\ldots,t^N), t \in [a,b]$. Then $\sum_{k=1}^{(N+3)/2} \omega_k \gamma(t_k)$ is an interior point of the convex hull of the moment curve, where $a = t_1 < \cdots < t_{(N+3)/2}=b$ and $\sum_{k=1}^{(N+3)/2} \omega_k =1, \omega_k \in (0,1)$.
\end{lemma}
\begin{corollary}\label{corollary-int}
If a probaility measure $\mu$ on $[a,b]$ is supported on finitely many but at least $(N+3)/2$ points, and the support includes the boundary points $\{a,b\}$, then $\int \gamma(t) \, d\mu(t)$ is an interior point of the convex hull of the moment curve.
\end{corollary}
\begin{proof}
Notice that the integral is a convex combination of finitely many but more than $(N+3)/2$ points. Hence we can rewrite the convex combination as a convex combination of $(N+3)/2$ points, including $\gamma(a)$ and $\gamma(b)$, and a convex combination of the reminder. By the above Lemma \ref{lemma-int}, the first convex combination (when normilized, i.e.,  divided by the sum of its weights) is an interior point of the convex hull. Since a convex combination of an interior point and some other points in a convex set is again an interior point of the convex set, the value of the integral is an interior point of the convex hull of the moment curve.
\end{proof}

Now we turn to the proof of Corollary~\ref{k-bound}. Notice that 
\begin{align*}
\inf_{\mathrm{deg}(h) \leq d, \mathbb{E}(h(X)) =1} \mathbb{E} X|h(X)| \leq \inf_{\mathrm{deg}(Q) \leq \lfloor d/2 \rfloor, \mathbb{E} Q^{2}(X) =1} \mathbb{E} X Q^2(X).
\end{align*}
 We can obtain the exact value of the right hand-side  by the following 
 {\em quadrature argument}. Let $p = \lfloor d/2 \rfloor $. Consider the moment curve $\gamma(t) = (t,\ldots,t^{2p+1})$ on $[0,n]$, and a random variable $X \sim \mathrm{Bin}(n,1/2)$. By Corollary \ref{corollary-int}, $\mathbb{E} \gamma(X)$ is an interior point of the convex hull of $\gamma([0,n])$. By \cite{DIM24} there exist $0 < t_1 < \cdots < t_{p+1} < n$ such that 
\begin{align*}
\mathbb{E}\gamma(X) = \sum_{k=1}^{p+1} \omega_k \gamma(t_k),
\end{align*}
for some $\omega_k \in (0,1)$ with $\sum_{k=1}^{p+1} \omega_k = 1$. This implies that for any polynomial $P(t)$ of degree at most $2p+1$, we have that 
\begin{align*}
\mathbb{E} P(X) = \sum_{k=1}^{p+1} \omega_k P(t_k).
\end{align*}
In particular, choosing $P(t)$ to be $t^i K_{n,p+1}(t), i\in \{0,\ldots,p\}$, where $K_{n,p+1}$ is the Kravchuk polynomial of degree $p+1$ with respect to $\mathrm{Bin}(n,1/2)$,  we obtain
\begin{align*}
0=\mathbb{E} X^i K_{n,p+1}(X) = \sum_{k=1}^{p+1} \omega_k t_k^i K_{n,p+1}(t_k), \; i \in \{0,\ldots,p\}. 
\end{align*}
 Since the matrix $\{t_k^i\}_{0 \leq i \leq p, 1 \leq k \leq p+1}$ is a Vandermonde matrix with nonzero determinant, the linear system above only has the trivial solution, which means that $t_k$ are the roots of $p+1$ degree Kravchuk polynomial $K_{n,p+1}(t)$, and $t_1 = k_{n, p+1}$.

Let $h$ be a polynomial of degree at most $p$ with $\mathbb{E}h^{2}(X) \neq 0$, then 
\begin{align*}
\frac{\mathbb{E}X h^2(X)}{\mathbb{E} h^2(X)} = \frac{\sum_{k=1}^{p+1} \omega_k t_k h^2(t_k) }{\sum_{k=1}^{p+1} \omega_k h^2(t_k)} \geq t_1,
\end{align*}
since  $\frac{\mathbb{E}X h^2(X)}{\mathbb{E} h^2(X)}$ is a convex combination of $t_k's$.  Thus, we have 
\begin{align*}
\inf_{\mathrm{deg}(h) \leq p, \mathbb{E}h^{2}(X) \neq 0} \frac{\mathbb{E} X h^2(X)}{|\mathbb{E}h^2(X)|} \geq t_1.
\end{align*}
On the other hand, if we take $h$ to be a degree $p$ polynomial that vanishes at $t_k$'s for $k=2,\ldots,p+1,$ but does not vanish at $t_1$, then 
\begin{align*}
\frac{\mathbb{E}X h^2(X)}{\mathbb{E} h^2(X)} = t_1,
\end{align*}
which shows that the inequality above is an equality finishing the proof of the corollary.

\subsection{Proof of Corollary~\ref{pr-bound}}
By Theorem \ref{mtavarit}, it suffices to show that
\begin{align*}
\inf_{\mathrm{deg}(h)\leq d,  \mathbb{E}h(X) = 1} \mathbb{E}X|h(X)| \leq n-d, \quad X \sim Bin(n,1/2).
\end{align*}
Consider the following degree $d$ polynomial
\begin{align*}
\tilde{h}(x):= \prod_{k=n-d+1}^n (k-x), 
\end{align*}
and set $h(x) = \tilde{h}(x)/\mathbb{E} \tilde{h}(X)$. 
Notice that  $h \geq 0$ on $\{0, \ldots, n\}$ and $\mathbb{E} h(X) =1$. Since $h = 0$ on $\{n-d+1, \ldots, n\}$ it follows that $\mathbb{E} X |h(X)| \leq n-d$ which completes the proof.

\subsection{Proof of Corollary~\ref{genupper}}
By Corollary~\ref{pr-bound}, for $d=(1-\delta) n$ we have
\begin{align*}
\inf_{\mathrm{deg}(g)\leq d}\|f-g\|_{\infty} \leq 3 \delta s(f)
\end{align*}
which gives the first term in minimum  of the right hand-side of (\ref{upperbb1}). Bounding $k_{n,p}$ from above will give the second term.  By  \cite[Theorem 5.1]{ADGP13} (using the upper bound of $\kappa_{n,n}(p,N)$ in the reference with $p=1/2$; in our case $N$ is $n$, and $n$ is $p:=\lfloor d/2 \rfloor + 1$), we have that
\begin{equation}\label{root}
k_{n,p} \leq \frac{n}{2} - \sqrt{\frac{n-p+1}{2}}h_{p,1},  
\end{equation}
where $h_{p,1}$ is the largest root of the Hermite polynomial. By \cite[equation (6.32.5)]{S75} we have 
\begin{align*}
h_{n,p} =  (2p+1)^{\frac{1}{2}} - 6^{-\frac{1}{2}}(2p+1)^{-\frac{1}{6}}(i_1+\varepsilon_{n}) ,
\end{align*}
where $i_1$ is the lowest zero of the Airy function and it satisfies  $6^{-\frac{1}{2}} i_1 = 1.85575 \ldots$, and $\varepsilon_{n} \to 0$ as $n\to \infty$. 
Combining this result and inequality (\ref{root}), we obtain that
\begin{align*}
k_{n,p} \leq \frac{n}{2} - \sqrt{\frac{(n - p+1)(2p+1)}{2}} + c_{n} \sqrt{\frac{n - p+1}{2}}(2p+1)^{-1/6},
\end{align*}
where $c_{n}\geq 0$ and $c_{n} \to 1.85575\ldots$. We have 
\begin{align*}
    \frac{n}{2} - \sqrt{\frac{(n - p+1)(2p+1)}{2}} 
   = \frac{n}{2} - \sqrt{\frac{(n - \lfloor d/2 \rfloor)(2 \lfloor d/2 \rfloor + 3)}{2}}
   \leq \frac{n}{2} - \sqrt{(n-\frac{d}{2})\frac{d}{2}} = \frac{n}{2}(1-\sqrt{1-\delta^2}) \leq \frac{n}{2}\delta^2,
\end{align*}
where the first inequaltiy follows from $n - \lfloor d/2 \rfloor \geq n - d/2$ and $2 \lfloor d/2 \rfloor + 3 \geq d$, and the second inequality follows from $\sqrt{x} \geq x $ for all $x \in [0,1]$. For the remainder term, we have that 
\begin{align*}
    &c_{n} \sqrt{\frac{n - p+1}{2}}(2p+1)^{-1/6}
     = c_{n} \sqrt{\frac{n - \lfloor d/2 \rfloor}{2}}(2\lfloor d/2 \rfloor+3)^{-1/6} \leq \\
    & \frac{c_{n}}{\sqrt{2}} \sqrt{n-\frac{d}{2}+1}(d+3)^{-1/6} = \frac{c_{n}}{\sqrt{2}} \sqrt{(\frac{1+\delta}{2}+\frac{1}{n})n}[(1 - \delta +\frac{3}{n})n]^{-1/6} \leq \frac{c_{n}}{(1-\delta)^{1/6}} n^{1/3}.
\end{align*}
Combining the two estimates above, we have that
\begin{align*}
k_{n,p} \leq [0.5 \delta^2 + \frac{c_{n}}{(1-\delta)^{1/6}} n^{-2/3}]n.
\end{align*}
Thus by the Theorem \ref{k-bound}, we have that 
\begin{align*}
\inf_{\mathrm{deg}(g)\leq d}\|f-g\|_{\infty} \leq 3 (0.5 \delta^2 + \frac{c_{n}}{(1-\delta)^{1/6}} n^{-2/3}) s(f).
\end{align*}
Thus we obtain 
\begin{align*}
  \inf_{\mathrm{deg}(g)\leq d}\|f-g\|_{\infty} \leq 3\min\left\{\delta,  0.5 \delta^2 + \frac{c_{n}}{(1-\delta)^{1/6}} n^{-2/3}\right\} s(f) \leq C\min\{\delta, \max\{\delta^{2}, n^{-2/3}\}\}  
\end{align*}
finishing the proof of the corollary.

\subsection{Proof of Theorem~\ref{rocco}}\label{roc01}

Let $\langle \cdot, \cdot \rangle$ denote the dot product in $\mathbb{R}^{2^{n}}$. The vectors in $\mathbb{R}^{2^{n}}$ will be indexed by the set $\{0,1\}^{n}$. For example,  given $f :\{0,1\}^{n} \to \{-1,1\}$ we will write 
\begin{align*}
    \langle f, W_{S} \rangle = \sum_{x \in \{0,1\}^{n}} f(x)W_{S}(x).
\end{align*}
Fix any $\theta \in \{0,1\}^n$. We have 
$$
f_{>d}(\theta) 
= \frac{1}{2^n} \sum_{|S|>d} \ip{f}{W_{S}} W_{S}(\theta)
= \ip{f}{v_\theta}
\quad \text{where} \quad 
v_\theta(x) = \frac{1}{2^n} \sum_{|S|>d} W_{S}(\theta) W_{S}(x).
$$
Here we think about  $f$ as a vector in $\R^{2^n}$ with independent Rademacher coefficients, 
and  $v_\theta$ as a fixed vector in $\R^{2^n}$. 
Hoeffding's inequality gives
\begin{align*}
\| \ip{f}{v_\theta} \|_\psitwo 
\lesssim \| v_\theta \|_2
&= \frac{1}{2^n} \left( \sum_{|S|>d} \|W_{S}\|_2^2 \right)^{1/2}
\quad \text{(by orthogonality of monomials $W_{S}(x)$)}\\
&= \sqrt{2^{-n} \binom{n}{>d}}
\quad \text{(since $\| W_{S} \|_2^2 = 2^n$ for each $S$)}.
\end{align*}
It follows from definition of $\psitwo$ norm  \cite[Definition 2.5.6]{Vershynin18} that 
$$
\Pr{|\ip{f}{v_\theta}| > C\sqrt{2^{-n} \binom{n}{>d} n}} \le 2^{-2n}.
$$
Taking a union bound yields
$$
\Pr{\exists \theta \in \{0,1\}^n:\; |\ip{f}{v_\theta} > C\sqrt{2^{-n} \binom{n}{>d} n}} 
\le 2^n \cdot 2^{-2n} = 2^{-n}.
$$
Thus, with probability at least $1-2^{-n}$, we have
$$
\|f_{>d}\|_\infty \lesssim \sqrt{2^{-n} \binom{n}{>d} n}.
$$	
By Hoeffding's inequality, the quantity on the right hand side is bounded by $1/n^{10}$ whenever $d \ge n/2 + C\sqrt{n \log n}$.

\subsection{Proof of Theorem~\ref{thres1}}\label{cnobili}
Fix $d \in \{1,\ldots,n\}$. Polynomials of degree at most $d$ form the linear space
$$
L_{d} \coloneqq \left\{ g:\; \deg(g) \le d \right\}
\quad \text{with} \quad
\dim(L_{d}) = \binom{n}{\le d} \eqqcolon m.
$$
For each $x \in \{0,1\}^n$, consider the central hyperplane 
$$
H_x \coloneqq \{g \in L_d:\; g(x)=0\}.
$$
There are $p \coloneqq 2^n$ such hyperplanes, which partition the space $L_{d}$ into regions. Each region is formed by polynomials $g$ that have the same sign pattern $\sign(g)$. Thus, the number of all possible sign patterns $\sign(g)$ equals the number of regions. On the other hand, in order to have $\inf_{g \in L_{d}} \| f-g \|_\infty < 1$, one must have $\sign(g)=f$. Thus, the number of Boolean functions that satisfy this inequality equals the number of sign patterns, which equals the number of regions. 
	
Schlaffli's formula \cite[Lemma2.2]{BV19} states that for any arrangement of $p$ central hyperplanes in $\R^m$, the number of regions is bounded by 
$$
r(p,m) \coloneqq 2 \binom{p-1}{\le m-1}.
$$
If $d \leq n/2$ then 
$$
m = \binom{n}{\le d} \le 2^n/2 = p/2.
$$ 
Thus $m-1 < (p-1)/2$, hence 
$r(p,m) \le 2 \cdot 2^{p-1}/2 = 2^{2^n}/2$. 
We showed that the number of Boolean functions that satisfy $\inf_{g \in L} \| f-g \|_\infty < 1$ is at most $2^{2^n}/2$.

\section*{Acknowledgements}
	P.I. was supported in part by NSF CAREER-DMS-2152401. R.V. was supported from NSF DMS-1954233, NSF DMS-2027299, U.S. Army 76649CS, and NSF+Simons Research Collaborations on the Mathematical and Scientific Foundations of Deep Learning. We thank Alexandros Eskenazis for helpful comments that improved the exposition in this paper.

\end{document}